\documentclass[11pt]{amsart}
\usepackage{amssymb}
\usepackage{amsmath}
\usepackage{amstext}

\usepackage{amsthm}
\usepackage{epsf}
\usepackage[scanall]{psfrag}
\usepackage{graphicx, color}

\def\Thmname{Theorem}
\def\Propname{Proposition}
\def\Lemmaname{Lemma}
\def\Definitionname{Definition}
\def\SSQ{
\left[%
\begin{array}{cc}
  e & f \\
  h & g \\
\end{array}%
\right] }
\def\GE{\mathcal{G}(E)}
\def\NE{\mathcal{N}(E)}

\newtheorem{Thm}{\Thmname}[section]

\newtheorem{Lem}[Thm]{\Lemmaname}

\newtheorem{Cor}[Thm]{Corollary}

\def\J{\mathrel{{\mathcal J}}} 

\def\R{\mathrel{{\mathcal R}}} 
\def\L{\mathrel{{\mathcal L}}} 
\def\H{\mathrel{{\mathcal H}}} 

\def\e<{\leq _{E}}

\def\l{\mathrel{\leq _{\L}}}
\def\r{\mathrel{\leq _{\R}}}

\def\lb{\mathrel{<_{\L}}}
\def\lb{\mathrel{<_{\L}}}

\font\petite=cmmi10 at 8pt
\def\malce{\mathbin{\hbox{$\bigcirc$\rlap{\kern-9pt\raise0,75pt\hbox{\petite
m}}}}}

\title{Subgroups of free
idempotent generated semigroups need not be free}
\author{Mark Brittenham}
\address{Department of Mathematics \\
University of Nebraska \\
Lincoln, Nebraska 68588-0323 \\
USA} \email{mbrittenham2@math.unl.edu}

\author{Stuart~W.~Margolis}
\address{Department of Mathematics \\
Bar Ilan University \\
52900 Ramat Gan \\
Israel} \email{margolis@math.biu.ac.il}

\author{John Meakin}
\address{Department of Mathematics \\
University of Nebraska \\
Lincoln, Nebraska 68588-0323 \\
USA} \email{jmeakin@math.unl.edu}

\thanks{The first author acknowledges support from NSF Grant
DMS-0306506. The second author acknowledges support from the
Department of Mathematics, University of Nebraska-Lincoln}

\begin{document}

\maketitle

\section{Introduction}

Let $S$ be a semigroup with set $E(S)$ of idempotents, and let
$\langle E(S) \rangle$ denote the subsemigroup of $S$ generated by
$E(S)$. We say that $S$ is an {\em idempotent generated} semigroup
if $S = \langle E(S) \rangle$. Idempotent generated semigroups
have received considerable attention in the literature. For
example, an early result of J. A. Erd\"os \cite{Erdos} proves that
the idempotent generated part of the semigroup of $n \times n$
matrices over a field consists of the identity matrix and all
singular matrices. J. M. Howie \cite{Howfulltr} proved a similar
result for the full transformation monoid on a finite set and also
showed that every semigroup may be embedded in an idempotent
generated semigroup. This result has been extended in many
different ways, and many authors have studied the structure of
idempotent generated semigroups. Recently, Putcha \cite{Putch2}
gave necessary and sufficient conditions for a reductive linear
algebraic monoid to have the property that every non-unit is a
product of idempotents, significantly generalizing the results of
J.A. Erd\"os mentioned above.


In 1979 K.S.S. Nambooripad \cite{Namb} published an influential
paper about the structure of (von Neumann) regular semigroups.
Nambooripad observed that the set $E(S)$ of idempotents of a
semigroup carries a certain structure (the structure of a
``biordered set", or a ``regular biordered set" in the case of
regular semigroups) and he provided an axiomatic characterization
of (regular) biordered sets in his paper. If $E$ is a regular
biordered set, then there is a free object, which we will denote
by $RIG(E)$, in the category of regular idempotent generated
semigroups with biordered set $E$. Nambooripad  showed how to
study $RIG(E)$ via an associated groupoid ${\mathcal N}(E)$. There
is also a free object, which we will denote by $IG(E)$, in the
category of idempotent generated semigroups with biordered set $E$
for an arbitrary (not necessarily regular) biordered set $E$.

In the present paper we provide a topological approach to
Nambooripad's theory by associating a $2$-complex $K(E)$ to each
regular biordered set $E$. The fundamental groupoid of the
$2$-complex $K(E)$ is Nambooripad's groupoid ${\mathcal N}(E)$.
Our concern in this paper is in analyzing the structure of the
maximal subgroups of $IG(E)$ and $RIG(E)$ when $E$ is a regular
biordered set. It has been conjectured that these subgroups are
free \cite{McElw}, and indeed there are several papers in the
literature (see for example, \cite{Pastijn3}, \cite{NP1},
\cite{McElw}) that prove that the maximal subgroups are free for
certain classes of biordered sets. The main result of this paper
is to use these topological tools to give the first example of
non-free maximal subgroups in free idempotent generated semigroups
over a biordered set.
We give an example of a regular biordered set $E$ associated to a
certain combinatorial configuration such that $RIG(E)$ has a
maximal subgroup isomorphic to the free abelian group of rank 2.

\section
{\bf Preliminaries on Biordered Sets and Regular Semigroups}

One obtains significant information about a semigroup by studying
its ideal structure. Recall that if $S$ is a semigroup and $a,b
\in S$ then the Green's relations ${\mathcal R, \mathcal L,
\mathcal H, \mathcal J}$ and $\mathcal D$ are defined by $a
{\mathcal R} b $ if and only if $aS^{1} = bS^{1}$, $a {\mathcal L}
b$ if and only if $S^{1}a = S^{1}b$, $a {\mathcal J} b$ if and
only if $S^{1}aS^{1} = S^{1}bS^{1}$, ${\mathcal H} = {\mathcal R}
\cap {\mathcal L}$ and ${\mathcal D} = {\mathcal R} \circ
{\mathcal L} = {\mathcal L} \circ {\mathcal R}$, so that
${\mathcal D}$ is the join of ${\mathcal R}$ and $ {\mathcal L}$
in the lattice of equivalence relations on $S$. The corresponding
equivalence classes of an element $a \in S$ are denoted by $R_{a},
L_{a}, H_{a}, J_{a}$ and $D_{a}$ respectively. Recall also that
there are quasi-orders defined on $S$ by $a \r b$ if $aS^{1}
\subseteq bS^{1}$, and $a \l b$ if $S^{1}a \subseteq S^{1}b$. As
usual, these induce partial orders on the set of $\R$-classes and
$\L$-classes respectively. The restrictions of these quasi-orders
to $E(S)$ will be denoted by ${\omega}^{r}$ and ${\omega}^{l}$
respectively in this paper, in accord with the notation in
Nambooripad's paper \cite{Namb}. It is easy to see that if $e$ and
$f$ are idempotents of $S$ then $e \, \, {\omega}^{r} \, \, f$
(i.e. $eS \subseteq fS$) if and only if $e = fe$, that $e \, \,
{\omega}^{l} \, \, f$ if and only if $e = ef$, that $e \, \,
{\mathcal R} \, \, f$ if and only if $e = fe$ and $f = ef$, and
that $e \, \, {\mathcal L} \, \, f$ if and only if $e = ef$ and $f
= fe$.

Let $e$ be an idempotent of a semigroup $S$. The set $eSe$ is a
submonoid in $S$ and is the largest submonoid (with respect to
inclusion) whose identity element is $e$. The group of units $G_e$
of $eSe$, that is the group of elements of $eSe$ that have two
sided inverses with respect to $e$, is the largest subgroup of $S$
(with respect to inclusion) whose identity is $e$ and is called
the maximal subgroup of $S$ at $e$.

Recall also that if $e$ and $f$ are idempotents of $S$ then the
natural partial order on $E(S)$ is defined by $e \,\, {\omega}
\,\,  f$ if and only if $ef = fe = e$. Thus $\omega=\omega^{r}
\cap \omega^{l}$. An element $a \in S$ is called {\it regular} if
$a \in aSa$: in that case there is at least one {\it inverse} of
$a$, i.e. an element $b$ such that $a = aba$ and $b = bab$. Note
that regular semigroups have in general many idempotents: if $a$
and $b$ are inverses of each other, then $ab$ and $ba$ are both
idempotents (in general distinct). Standard examples of regular
semigroups are the semigroup of all transformations on a set (with
respect to composition of functions) and the semigroup of all $n
\times n$ matrices over a field (with respect to matrix
multiplication).

We recall the basic properties of the very important class of
completely 0-simple semigroup. A semigroup $S$ (with 0) is
(0)-simple if ($S^{2}\neq 0$ and) its only ideal is $S$ ($S$ and
{0}). A (0)-simple semigroup $S$ is completely (0)-semigroup if
$S$ contains an idempotent and every idempotent is (0)-minimal in
the natural partial order of idempotents defined above. It is a
fundamental fact that every finite (0)-simple semigroup is
completely (0)-simple.

Let $S$ be a completely 0-simple semigroup. The Rees theorem
\cite{CP,Lall} states that $S$ is isomorphic to a regular Rees
matrix semigroup $M^{0}(A,G,B,C)$ and conversely that every such
semigroup is completely (0)-simple. Here $A (B)$ is an index set
for the $\R (\L)$-classes of the non-zero $\J$-class of $S$ and
$C:B \times A\rightarrow G^{0}$ is a function called the structure
matrix. $C$ has the property that for each $a \in A$ there is a $b
\in B$ such that $C(b,a) \neq 0$ and for each $b \in B$ there is
an $a \in A$ such that $C(b,a) \neq 0$. We always assume that $A$
and $B$ are disjoint. The underlying set of $M^{0}(A,G,B,C)$ is $A
\times G \times B \cup \{0\}$ and the product is given by
$(a,g,b)(a',g',b')=(a,gC(b,a')g',b')$ if $C(b,a')\neq 0$ and 0
otherwise.

We refer the reader to the books of Clifford and Preston \cite{CP}
or Lallement \cite{Lall} for standard ideas and notation about
semigroup theory.

An {\it $E$-path} in a semigroup $S$ is a sequence of idempotents
$(e_{1},e_{2}, \ldots , e_{n})$ of $S$ such that $e_{i} \, \,
({\mathcal R} \cup {\mathcal L}) \, \, e_{i+1}$ for all $i = 1,
\ldots n-1$. This is just a path in the graph $(E, {\mathcal R}
\cup {\mathcal L})$: the set of vertices of this graph is the set
$E$ of idempotents of $S$ and there is an edge denoted $(e,f)$
from $e$ to $f$ for $e,f \in E$ if $e {\mathcal R} f$ or $e
{\mathcal L} f$. One can introduce an equivalence relation on the
set of $E$-paths by adding or removing ``inessential" vertices: a
vertex (idempotent) $e_{i}$ of a path $(e_{1}, e_{2} \ldots ,
e_{n})$ is called {\it inessential} if $e_{i-1} \, \, {\mathcal R}
\, \, e_{i} \, \, {\mathcal R} \, \, e_{i+1}$ or $e_{i-1} \, \,
{\mathcal L} \, \, e_{i} \, \, {\mathcal L} \, \, e_{i+1}$.
Following Nambooripad \cite{Namb}, we define an {\it $E$-chain} to
be the equivalence class of an $E$-path relative to this
equivalence relation. It can be proved \cite{Namb} that each
$E$-chain has a unique canonical representative of the form
$(e_{1}, e_{2}, \ldots , e_{n})$ where every vertex is essential.
We will often abuse notation slightly by identifying an $E$-chain
with its canonical representative.

The set ${\mathcal G}(E)$ of $E$-chains forms a {\it groupoid}
with set $E$ of objects (identities) and with an $E$-chain
$(e_{1},e_{2}, \ldots , e_{n})$ viewed as a morphism from $e_{1}$
to $e_{n}$. The product $C_{1}C_{2}$ of two $E$-chains $C_{1} =
(e_{1},e_{2}, \ldots , e_{n})$ and $C_{2} = (f_{1},f_{2}, \ldots,
f_{m})$ is defined and equal to the canonical representative of
$(e_{1}, \ldots , e_{n}, f_{1}, \ldots f_{m})$ if and only if
$e_{n} = f_{1}$: the inverse of $(e_{1},e_{2}, \ldots , e_{n})$ is
$(e_{n}, \ldots , e_{2},e_{1})$. We refer the reader to
\cite{Namb} for more detail.

For future reference we give a universal characterization of
${\mathcal G}(E)$ in the category of small groupoids. Every
equivalence relation $R$ on a set $X$ can be considered to be a
groupoid with objects $X$ and arrows the ordered pairs of $R$.
There are obvious notions of free products and free products with
amalgamations in the category of small groupoids. See
\cite{Higgoids} for details. Clearly the objects of any groupoid
form a subgroupoid whose morphisms are the identities. We will
identify the objects of a groupoid as this subgroupoid and call it
the trivial subgroupoid. The proof of the following theorem
appears in \cite{Namb}.

\begin{Thm}\label{amalg}
Let $S$ be a semigroup with non-empty set of idempotents $E$. Then
${\mathcal G}(E)$ is isomorphic to the free product with
amalgamation $\L\ast_{E}\R$ in the category of small groupoids.
\end{Thm}

As mentioned above, we are considering $E$ to be the trivial
subgroupoid of $\GE$.


It is easy to see from the characterizations of ${\mathcal R}$ and
$\mathcal L$ above that if $(f_{1},f_{2}, \ldots ,f_{m})$ is the
canonical representative equivalent to an $E$-path
\\ $(e_{1},e_{2},\ldots , e_{n})$, then $e_{1}e_{2} \ldots e_{n} =
f_{1}f_{2} \ldots f_{m}$ in $S$, since $efg = eg$ if $e {\mathcal R}
f  {\mathcal R}  g$ or $e  {\mathcal L}  f  {\mathcal L} g$.
Standard results of Miller and Clifford \cite{CP} imply that
\\$e_{1}  {\mathcal R} e_{1}e_{2} \ldots e_{n}  {\mathcal L} \,
\, e_{n}.$

In 1972, D.G. Fitzgerald \cite{FitzG} proved the following basic
result about the idempotent generated subsemigroup of any
semigroup.

\begin{Thm}\label{fitz}
Let $S$ be any semigroup with non-empty set $E = E(S)$ of
idempotents and let $x$ be a regular element of $\langle E(S)
\rangle$. Then $x$ can be expressed as a product of idempotents $x
= e_{1}e_{2} \ldots e_{n}$ in an $E$-path \\ $(e_{1}, e_{2},
\ldots , e_{n})$ of $S$, and hence as a product of idempotents in
an $E$-chain. If $S$ is regular, then so is $<E(S)>$.
\end{Thm}

In 1979, Nambooripad introduced the notion of a {\it biordered
set} as an abstract characterization of the set of idempotents $E$
of a semigroup $S$ with respect to certain basic products that are
forced to be idempotents. We give the details that will be needed
in this paper.


Recall that if $e,f \in E=E(S)$ for some semigroup $S$ then  $e \,
\, {\omega}^{r} \, \, f$ if and only if $fe = e$, and $e \, \,
{\omega}^{l} \, \, f$ if and only if $ef = e$. In the former case,
$ef$ is an idempotent that is $\mathcal R$ -related to $e$ and $ef
\,\, {\omega} \,\, f$ in the natural order on $E$: similarly, in
the latter case, $fe$ is an idempotent that is $\mathcal
L$-related to $e$ and $fe \,\, {\omega} \,\, f$. Thus in each case
both products $ef$ and $fe$ are defined within $E$, i.e. such
products of idempotents must always be idempotent. Products of
these type are referred to as {\it basic products}. The partial
algebra $E$ with multiplication restricted to basic products is
called the {\it biordered set} of $S$.

Nambooripad \cite{Namb} characterized the partial algebra of
idempotents of a (regular) semigroup with respect to these basic
products axiomatically. We refer the reader to Nambooripad's article
\cite{Namb} for the details. The axioms are complicated but do arise
naturally in mathematics. For example, Putcha proved that pairs of
opposite parabolic subgroups of a finite group of Lie type have the
natural structure of a biordered set \cite{Putchbook}. We will need
one more concept, the sandwich set $S(e,f)$ of two idempotents $e,f$
of $S$.

If $e,f$ are (not necessarily distinct) idempotents of a semigroup
$S$, then $S(e,f)=\{h \in E|ehf=ef, fhe=h\}$ is called the
sandwich set of $e$ and $f$ (in that order). It is straightforward
to prove that if $h \in S(e,f)$, then $h$ is an inverse of $ef$.
In particular, $S(e,f)$ is non-empty for any $e,f$ if $S$ is a
regular semigroup. Nambooripad also gave an order theoretic
definition of the sandwich set, but we will not need that in this
paper.

As mentioned above, Nambooripad gave a definition of a biordered
set as a partial algebra satisfying a collection of axioms. We
don't need the details of these axioms because of the following
theorems. He called a biordered set {\it regular} if the
(axiomatically defined) sandwich set of any pair of idempotents is
non-empty.

\begin{Thm} (Nambooripad \cite{Namb})
The set $E$ of idempotents of a regular semigroup is a regular
biordered set relative to the basic products in $E$. Conversely,
every regular (axiomatically defined) biordered set arises as the
biordered set of idempotents of some regular semigroup.
\end{Thm}

This was extended to non-regular semigroups and non-regular
biordered sets by Easdown \cite{Eas1}. We will give a more precise
statement of Easdown's result in the next section.

\section
{\bf Free idempotent generated semigroups on biordered sets}

If $E$ is a biordered set we denote by $IG(E)$ the semigroup with
presentation

\medskip

$IG(E) = \langle E: e^{2} = e$ for all $e \in E$ and $e.f = ef$ if
$ef$ is a basic product in $E \rangle $.

\medskip

If $E$ is a regular biordered set, then we define

\medskip

$RIG(E) = \langle E: e^{2} = e$ for all $e \in E$ and $e.f = ef$
if $ef$ is a basic product in $E$ and $ef = ehf$ for all $e,f \in
E$ and $h \in S(e,f) \rangle$

\medskip

The semigroup $IG(E)$ is called the {\it free idempotent generated
semigroup on $E$} and the semigroup $RIG(E)$ is called the {\it
free regular idempotent generated semigroup on $E$}. This
terminology is justified by the following results of Easdown
\cite{Eas1}, Nambooripad \cite{Namb} and Pastijn \cite{Past2}.

\begin{Thm}\cite{Eas1}\label{Eas}
The biordered set of idempotents of $IG(E)$ is $E$. In particular,
every biordered set is the biordered set of some semigroup. If $S$
is any idempotent generated semigroup with biordered set of
idempotents isomorphic to $E$ then the natural map $E \rightarrow
S$ extends uniquely to a homomorphism $IG(E) \rightarrow S$.
\end{Thm}

\begin{Thm}\cite{Namb, Past2}\label{Nam}
If $E$ is a regular biordered set then $RIG(E)$ is a regular
semigroup with biordered set of idempotents $E$. If $S$ is any
regular idempotent generated semigroup with biordered set biorder
isomorphic to $E$, then the natural map $E \rightarrow S$ extends
uniquely to a homomorphism $RIG(E) \rightarrow S$.
\end{Thm}

There is an obvious natural morphism ${\phi} : IG(E) \rightarrow
RIG(E)$ if $E$ is a regular biordered set. However, we remark that
this is not an isomorphism, and the semigroups $IG(E)$ and
$RIG(E)$ can be very different when $E$ is a regular biordered
set. Also, the regular elements of $IG(E)$ do not form a
subsemigroup in general, even if $E$ is a regular biordered set.

The following simple examples illustrate these facts.

\medskip

\noindent {\bf Example 1.} Let $E$ be the (non-regular) biordered
set consisting of two idempotents $e$ and $f$ with trivial
quasi-orders ${\omega}^r$ and ${\omega}^l$. Clearly the rules
$e^{2} \rightarrow e, f^{2} \rightarrow f$ constitute a
terminating confluent rewrite system for the semigroup $IG(E)$.
Canonical forms for words in $IG(E)$ are of the form $efef \ldots
e$ or $efef \ldots f$ or $fe fe \ldots f$ or $fefe \ldots e$.
Clearly $IG(E)$ is an infinite semigroup with exactly two
idempotents ($e$ and $f$).

\noindent {\bf Example 2.} Let $F$ be the biordered set $E$ above
with a zero $0$ adjoined. Thus $F$ is a three-element semilattice,
freely generated as a semilattice by $e$ and $f$. It is easy to
see that $RIG(F) = F$ since $ef = e0f = fe = f0e = 0$ from the
presentation for $RIG(F)$ and since $0\in S(e,f)$. But $IG(F)$ is
$IG(E)^{0}$, where $IG(E)$ is the semigroup in Example 1. Thus
$IG(F)$ is infinite, but $RIG(F)$ is finite.

\medskip
We will give more information about the relationship between
$IG(E)$ and $RIG(E)$, for $E$ a regular biordered set, at the end
of this section. In particular, we will show that the regular
elements of $IG(E)$ are in one-one correspondence with the
elements of $RIG(E)$ (even though the regular elements of $IG(E)$
do not necessarily form a subsemigroup of $IG(E)$).

Nambooripad studied the free regular idempotent generated
semigroup semigroup $RIG(E)$ on a regular biordered set via his
general theory of ``inductive groupoids" in \cite{Namb}. If $S$ is
a regular semigroup, then Nambooripad introduced an associated
groupoid ${\mathcal N}(S)$ (that we refer to as the {\it
Nambooripad groupoid of $S$)} as follows. The set of objects of
${\mathcal N}(S)$ is the set $E = E(S)$ of idempotents of $S$. The
morphisms of ${\mathcal N}(S)$ are of the form $(x,x')$ where $x'$
is an inverse of $x$: $(x,x')$ is viewed as a morphism from $xx'$
to $x'x$ and the composition of morphisms is defined by
$(x,x')(y,y') = (xy,y'x')$ if $x'x = yy'$ (and undefined
otherwise). With respect to this product, ${\mathcal N}(S)$
becomes a groupoid, which in fact is endowed with much additional
structure, making it an {\it inductive groupoid} in the sense of
Nambooripad \cite{Namb}. An inductive groupoid is an ordered
groupoid whose identities (objects) admit the structure of a
regular biordered set $E$, and which admits a way of evaluating
products of idempotents in an $E$-chain as elements of the
groupoid.  There is an equivalence between the category of regular
semigroups and the category of inductive groupoids. We refer the
reader to Nambooripad's paper \cite{Namb} for much more detail. In
particular, it follows easily from Nambooripad's results that the
maximal subgroup of $S$ containing the idempotent $e$ is
isomorphic to the local group of ${\mathcal N}(S)$ based at the
object (identity) $e$ (i.e. the group of all morphisms from $e$ to
$e$ in ${\mathcal N}(S)$).

In his paper \cite{Namb}, Nambooripad also showed how to construct
the inductive groupoid ${\mathcal N}(RIG(E))$ associated with the
free regular idempotent generated semigroup on a regular biordered
set $E$ directly from the groupoid of $E$-chains of $E$. We review
this construction here.

Let $E$ be a regular biordered set. An {\it $E$-square} is an $E$
-path $(e,f,g,h,e)$ with $e\R f \L g \R h \L e$ or $(e,h,g,f,e)$
with $e \L h \R g \L f \R e$. We draw the square as: $
\left[%
\begin{array}{cc}
  e & f \\
  h & g \\
\end{array}%
\right] $. An $E$-square is degenerate if it is of one of the
following three types:
\begin{center}
$
\left[%
\begin{array}{cc}
  e & e \\
  e & e \\
\end{array}%
\right]$  $\left[%
\begin{array}{cc}
  e & f \\
  e & f \\
\end{array}%
\right] $
$\left[%
\begin{array}{cc}
  e & e \\
  f & f \\
\end{array}%
\right] $

Unless mentioned otherwise, all $E$-squares will be
non-degenerate.
\end{center}

An idempotent $t=t^2 \in E$ {\em left to right singularizes} the
$E$-square $
\left[%
\begin{array}{cc}
  e & f \\
  h & g \\
\end{array}%
\right] $ if $te=e, th=h, et=f$ and $ht=g$ where all of these
products are defined in the biordered set $E$. Right to left, top
to bottom and bottom to top singularization is defined similarly
and we call the $E$-square {\em singular} if it has  a
singularizing idempotent of one of these types. Note that since
$te=e \in E$ if and only if $e\omega^{r}t$, all of these products
can also be defined in terms of the order structure as well.

The importance of singular $E$-squares is given by the next lemma.

\begin{Lem}\label{Trivsing} Let $\SSQ$ be a singular $E$-square in
a semigroup $S$.Then the product of the elements in the $E$-cycle
$(e,f,g,h,e)$ satisfies $efghe=e$.
\end{Lem}

\begin{proof}

Let $t=t^2$ left to right singularize the $E$-square $\SSQ$. Then
in any idempotent generated semigroup with biordered set $E$,
$efghe=fh$ follows from the basic $\R$ and $\L$ relations of $E$.
Furthermore, $fh=eth=eh=e$ which follows from the definition of
left to right singularization. The other cases of singularization
are proved similarly.
\end{proof}

In order to build the inductive groupoid of $RIG(E)$, we must
therefore identify any singular $E$-cycle of $\GE$ from an
idempotent $e$ to itself with $e$. This is because any inductive
groupoid with biordered set $E$ is an image of $\GE$ by
Nambooripad's theory \cite{Namb}. This leads to the following
definition. For two $E$-chains $C = (e_{1},e_{2}, \ldots , e_{n})$
and $C' = (f_{1},f_{2}, \ldots , f_{m})$ define $C \rightarrow C'$
if there are $E$-chains $C_{1}$ and $C_{2}$ and a singular
$E$-square $\gamma$ such that $C = C_{1}C_{2}$ and $C' =
C_{1}{\gamma}C_{2}$ and let $\sim$ denote the equivalence relation
on ${\mathcal G}(E)$ induced by $\rightarrow$. The next theorem
follows from \cite{Namb} Theorem 6.9, 6.10 and ensures that the
quotient groupoid ${\mathcal G}(E)/{\sim} $ defined above has an
inductive structure and is isomorphic to the inductive groupoid of
$RIG(E)$.

\begin{Thm}(Nambooripad \cite{Namb}) \label{N(E)}
If $E$ is a regular biordered set, then $ {\mathcal N}(RIG(E))
\cong {\mathcal G}(E)/{\sim} $.
\end{Thm}

It is convenient to provide a topological interpretation of this
theorem of Nambooripad. We remind the reader that just as groups
are presented by a set of generators and a set of words over the
generating set as relators (giving the group as a quotient of the
free group on the generating set), groupoids are presented by a
graph and a set of cycles in the graph as relators (giving the
groupoid as a quotient of the free groupoid on the graph). See
\cite{Higgoids} for more details.

It follows from Theorem \ref{amalg} and Theorem \ref{N(E)} that we
have the following presentation for $ {\mathcal N}(RIG(E)) \cong
{\mathcal G}(E)/{\sim} $.

{\bf Generators:} The graph with vertices $E$ and edges the
relation $\R \cup \L$.

{\bf Relators:} There are two types of relators:

\begin{enumerate}
    \item {$((e,f),(f,g),(g,e))=1_e$ if $e\R f\R g$ or $e \L f \L g$}
    \item {$((e,f),(f,g),(g,h),(h,e))=1_e$ if $\SSQ$ is a singular
$E$-square}.
\end{enumerate}

We will always assume that there are no trivial relators in the
list above. This means that for relators of type (1) all three
elements $e,f,g$ are distinct and for relators of type (2), all
four elements $e,f,g,h$ are distinct.

 If $E$ is a regular biordered set  we associate a 2-complex
$K(E)$ which is the analogue of the presentation complex of a
group presentation. The $1$-skeleton of $K(E)$ is the graph
$(E,{\mathcal R} \cup {\mathcal L})$ described above. Since $\R$
and $\L$ are symmetric relations we consider the underlying graph
to be undirected in the usual way. The $2$-cells of $K(E)$ are of
the following types:

(1) if $e \, \, {\mathcal R} \, \, f \, \, {\mathcal R} \, \, g$
or $e \, \, {\mathcal L} \, \, f \, \, {\mathcal L} \, \, g$ for
$e,f,g \in E$ then there is a $2$-cell with boundary edges $(e,f),
(f,g), (g,e)$.

(2) all singular $E$-squares bound $2$-cells.

We note that our 2-complexes are combinatorial objects and we follow
the notation of \cite{Rotman}, \cite{Spanier}.

We denote the fundamental groupoid of a $2$-complex $K$ by
${\pi}_{1}(K)$: the fundamental group of $K$ based at $v$ will be
denoted by ${\pi}_{1}(K,v)$. The following corollary is an
immediate consequence of Nambooripad's work and the definition of
the fundamental groupoid of a $2$-complex (see, for example,
\cite{Higgoids}).

\begin{Cor} If $E$ is a regular biordered set, then
${\pi_{1}}(K(E)) \cong {\mathcal G}(E)/{\sim} $ and hence
${\pi_{1}}(K(E)) \cong {\mathcal N}(RIG(E))$.
\end{Cor}

It follows that the maximal subgroup of $RIG(E)$ containing the
idempotent $e$ is isomorphic to the fundamental group of $K(E)$
based at $e$. The next theorem shows that there is a one to one
correspondence between regular elements of $IG(E)$ and $RIG(E)$ if
$E$ is a regular biordered set and that for every $e \in E$, the
maximal subgroup at $e$ in $IG(E)$ is isomorphic to the maximal
subgroup at $e$ in $RIG(E)$.

\begin{Thm}\label{IG=RIG}
Let $E$ be a regular biordered set. Then the natural map ${\phi} :
IG(E) \rightarrow RIG(E)$ is a bijection when restricted to the
regular elements of IG(E). That is, for each element $r \in
RIG(E)$ there exists a unique regular element $s \in IG(E)$ such
that ${\phi}(s) = r$. In particular, the maximal subgroups of
$IG(E)$ and $RIG(E)$ are isomorphic.
\end{Thm}

\begin{proof}

It follows from Fitzgerald's theorem, Theorem \ref{fitz} that every
element of $RIG(E)$ is the product of the elements in an $E$ chain.
But it follows from the Clifford-Miller theorem \cite{CP} that the
product of an element in an $E$-chain is a regular element in any
idempotent generated semigroup with biordered set $E$. It follows
immediately that $\phi$ restricts to a surjective map from the
regular elements of $IG(E)$ to $RIG(E)$.

If $u$ and $v$ are regular elements of $IG(E)$, then there are
$E$-chains \\
 $(e_{1},e_{2}, \ldots , e_{n})$ and $(f_{1},f_{2},
\ldots , f_{m})$ such that $u = e_{1}e_{2} \ldots e_n$ and \\
$v =
f_{1}f_{2} \ldots f_{m}$ in $IG(E)$. Suppose that ${\phi}(u) =
{\phi}(v)$. Clearly, on applying the morphism $\phi$, $e_{1}e_{2}
\ldots e_{n} = f_{1}f_{2} \ldots f_{m}$ in $RIG(E)$. We mentioned
previously that it follows from the Clifford-Miller theorem
\cite{CP} that $e_{1} {\mathcal R} f_{1}$ and $ e_{n} {\mathcal L}
f_{m}$. Thus without loss of generality, we may assume that $e_{1} =
f_{1}$ since $e_{1}f_{1}f_{2} \ldots f_{m} = f_{1}f_{2} \ldots
f_{m}$ in $IG(E)$, and similarly we may assume that $e_{n} = f_{m}$.
Applying \cite{Namb} Lemma 4.11 and Theorem \ref{N(E)}, it follows
that $(e_{1},e_{2}, \ldots ,e_{n}) \sim (f_{1}, f_{2}, \ldots
,f_{m})$. Thus it is possible to pass from $(e_{1},e_{2}, \ldots
,e_{n})$ to $(f_{1}, f_{2}, \ldots ,f_{m})$ by a sequence of
operations of two types:

(a) inserting or deleting paths of length 3 corresponding to $\R$
or $\L$ related idempotents; and

(b) inserting or deleting $E$-cycles corresponding to singular
$E$-squares.

Note that if $(e,f,g,h,e)$ is a singular $E$-square then $efghe =
e$ in any semigroup $S$ with biordered set $E$ by Lemma
\ref{Trivsing}. It follows easily that if $(g_{1},g_{2}, \ldots
,g_{p})$ is obtained from $(e_{1},e_{2}, \ldots ,e_{n})$ by one
application of an operation of type (a) or (b) above, then
$e_{1}e_{2} \ldots e_{n} = g_{1}g_{2} \ldots g_{p}$ in any
semigroup with biordered set $E$, and in particular this is true
in $IG(E)$. It follows by induction on the number of steps of
types (a) and (b) needed to pass from $(e_{1},e_{2}, \ldots
,e_{n})$ to $(f_{1},f_{2}, \ldots ,f_{m})$ that $u = e_{1}e_{2}
\ldots e_{n} = f_{1}f_{2} \ldots f_{m} = v$ in $IG(E)$, so $\phi$
is one-to-one on regular elements, as desired.

To prove the final statement of the theorem, note that elements of
the maximal subgroup of $IG(E)$ or $RIG(E)$ containing $e$ come
from
$E$-chains that start and end at $e$, since \\
$e_{1} \, \, {\mathcal R} \, \, e_{1}e_{2} \ldots e_{n} \, \,
{\mathcal L} \, \, e_{n}$ for any $E$-chain $(e_{1},e_{2}, \ldots
, e_{n})$. This shows that the map ${\phi}$ is surjective on
maximal subgroups: the first part of the theorem shows that it is
injective on maximal subgroups.

\end{proof}

\section{Connections between the Nambooripad Complex and the
Graham-Houghton Complex}

In this section we use the Bass-Serre theoretic methods of
\cite{HMM} to study the local groups of $\GE$ and $\NE$. The local
group of a groupoid $G$ at the object $v$ is the group of self
morphisms $G(v,v)$. For $\GE$ we give a rapid topological proof of
a result of Namboopripad and Pastijn \cite{NP1} who showed that
the local groups of $\GE$ are free groups. By applying \cite{HMM}
we are lead directly to the graphs considered by Graham and
Houghton \cite{Gr, Hough} for studying completely 0-simple
semigroups. We put a structure of a complex on top of the
Graham-Houghton graphs in order to have tools to study the vertex
subgroups of $\NE$, which by Theorem \ref{N(E)} and Theorem
\ref{IG=RIG} are the maximal subgroups of $IG(E)$ and $RIG(E)$
when $E$ is a regular biordered set.

Throughout this section, $E$ will denote a regular biordered set.
By Theorem \ref{Nam} $E$ is isomorphic to the biordered set of
idempotents of $RIG(E)$ and we will use this identification
throughout the section as well. Thus, we will refer to the
elements of $E$ as idempotents and talk about their Green classes
within $RIG(E)$. We have seen in Theorem \ref{amalg} that $\GE$
decomposes as the free product with amalgamation $\GE=\L \ast_{E}
\R$, where by abuse of notation, $E$ denotes the trivial
subgroupoid. Since $\L$ and $\R$ also have the same objects as
each other and as $E$, we can use the methods of \cite{HMM} to
study the maximal subgroup of $\GE$, since this paper was
concerned with amalgams of groupoids in which the intersection of
the two factors contains all the identity elements.

For every such amalgam of groupoids $G=A\ast_{U}B$, \cite{HMM}
associates a graph of groups in the sense of Bass-Serre Theory
\cite{Serre} whose connected components are in one to one
correspondence with the connected components of $G$ and such that
the fundamental group of a connected component is isomorphic to
the local group of the corresponding component of $G$.

First note that there is a one to one correspondence between the
$\L (\R)$ classes of $E$ and the $\L (\R)$ classes of $RIG(E)$.
This is because every $\L (\R)$ class of $RIG(E)$ has an
idempotent and the $\L (\R)$ relation restricted to idempotents
can be defined by basic products. We abuse notation by identifying
an $\L (\R)$ class of $E$ with the $\L (\R)$ class of $RIG(E)$
containing it.

We now describe explicitly the graph of groups associated to
$\GE$. For more details, see \cite{HMM}. The graph of groups
$\textsf{G}$ of $\GE$ consists of the following data: The set of
vertices is the disjoint union of the $\L$ and $\R$ classes of $E$
and its positive edges are the elements of $E$. If $e \in E$, its
initial edge is its $\L$-class and its terminal edge is its
$\R$-class. That is, there is a unique positive edge from an
$\L$-class $L$ to an $\R$-class $R$ if and only if the $\H$-class
$L \cap R$ of $RIG(E)$ contains an idempotent. Each vertex group
of $\textsf{G}$ is the trivial group. This is an exact translation
for $\GE$ of the graph of groups defined for an arbitrary amalgam
on page 46 of \cite{HMM}.

Since the vertex groups of $\textsf{G}$ are trivial, we can consider
$\textsf{G}$ to be a graph in the usual sense. Therefore its
fundamental group is a free group and we have the following theorem
of Namboopripad and Pastijn \cite{NP1}.

\begin{Thm}\label{free}
Every local subgroup of $\GE$ is a free group.
\end{Thm}

\begin{proof}

It follows from Theorem 3 of \cite{HMM} that for each element $e
\in E$ the local subgroup of $\GE$ at $e$ is isomorphic to the
fundamental group of $\textsf{G}$ based at the $\L$-class of $e$.
Since the latter group is free by the discussion above, the
theorem is proved.

\end{proof}

In the case that a connected component of $\GE$ has a finite
number of idempotents, the rank of the free group will be the
Euler characteristic of the corresponding component of
$\textsf{G}$, that is, the number of edges of the graph minus the
number of vertices plus 1. Thus if the connected component of $e
\in E$ of $\GE$ has $m$ $\R$-classes, $n$ $\L$-classes and $k$
idempotents, then the free group $\GE(e,e)$ has rank $k-(m+n)+1$.

All the calculations of maximal subgroups of $RIG(E)$ or $IG(E)$
that have appeared in the literature \cite{McElw, NP1, Past2} have
been restricted to cases of biordered sets that have no
non-degenerate singular squares. In this case it follows from
Theorem \ref{N(E)} that $\GE$ is isomorphic to $\NE$. Since the
local groups of $\NE$ are isomorphic to the maximal subgroups of
$RIG(E)$ we have the following result of Nambooripad and Pastijn
\cite{NP1}.

\begin{Thm}\label{nosing}

If $E$ is a biordered set that has no non-degenerate singular
squares, then every subgroup of $RIG(E)$ is free.

\end{Thm}

Nambooripad and Pastijn's proof of theorem \ref{nosing} uses
combinatorial word arguments. A topological proof of theorem
\ref{nosing} in the special case that the (not necessarily
regular) biordered set has no nontrivial biorder ideals was given
by McElwee \cite{McElw}. The graph that McElwee uses is the same
as ours in this case, but without reference to the general work of
\cite{HMM} or the connection with the Graham-Houghton graph
\cite{Gr, Hough} that we discuss below. There are a number of
interesting classes of regular semigroups whose biordered sets
have no non-degenerate singular squares including locally inverse
semigroups. See \cite{NP1} for more examples.

Connected components of the graph $\textsf{G}$ associated to $\GE$
defined above have arisen in the literature in connection with the
theory of finite 0-simple semigroups and in particular with the
theory of idempotent generated subsemigroups of finite 0-simple
semigroups. Finite idempotent generated 0-simple semigroups have the
property that all non-zero idempotents are connected by an
$E$-chain. This follows from the Clifford-Miller theorem \cite{CP}.
Thus the graph $\textsf{G}$ corresponding to the biordered set of a
finite 0-simple semigroup has a trivial component consisting of 0
and one other connected component. The graph defined independently
by Graham and Houghton \cite{Gr, Hough} associated to a finite
0-simple semigroup is exactly the graph that arises from Bass-Serre
theory associated to $\GE$ that we have defined above. Graham and
Houghton did not note the connection to Bass-Serre theory. A number
of papers have given connections between completely 0-simple
semigroups, the theory of graphs and algebraic topology \cite{Gr,
Hough}, \cite{Pollatch}. The monograph \cite{qtheory} gives an
updated version of these connections.

We now add 2-cells to $\textsf{G}$ of the graph associated to $\GE$,
one for each singular square $\SSQ$. Given this square and recalling
that the positive edges of $\textsf{G}$ are directed from the
$\L$-class of an idempotent to its $\R$-class we sew a 2-cell onto
$\textsf{G}$ with boundary $ef^{-1}gh^{-1}$. We call this 2-complex
the Graham-Houghton complex of $E$ and denote it by $GH(E)$.

We note two important properties of $GH(E)$. Its 1-skeleton is
naturally bipartite as each edge runs between an $\L$-class and an
$\R$-class. Furthermore $GH(E)$ is a square complex in that each of
its cells is a square bounded by a 4-cycle.

We now prove that the fundamental group of the connected component
of $GH(E)$ containing the vertex $\L_{e}$ of an idempotent $e\in
E$ is isomorphic to the fundamental group of the Nambooripad
complex $K(E)$ containing the vertex $e$. We will then be able to
use $GH(E)$ to compute the maximal subgroups of $RIG(E)$.



As we have seen above, the Nambooripad complex $K(E)$ has vertices
$E$, the idempotents of $S$, edges $(e,f)$ whenever $e\R f$ or
$e\L f$, and two types of two cells: one triangular 2-cell
$(e,f)(f,g)(g,e)$ for each unordered triple $(e,f,g)$ of distinct
elements satisfying $e\R f\R g$ or $e\L f\L g$, and one square
2-cell $(e,f)(f,g)(g,h)(h,e)$ for each non-degenerate singular
$E$-square $\SSQ$.

The Graham-Houghton complex $GH(E)$ has one vertex for each $\R$
or $\L$-class of $E$, an edge labelled by $e\in E$ between $\R_a$
and $\L_b$ if $e\in\R_a\cap\L_b$ (giving a bipartite graph), and
square 2-cells attached along $(e,f,g,h)$ when $\SSQ$ is a
non-degenerate singular $E$-square.

We now describe a sequence of transformations of complexes which
starts with $GH(E)$ and ends with $K(E)$. Each step, we shall see,
does not change the isomorphism class of the fundamental groups of
the complex. This will imply that $GH(E)$ and $K(E)$ have
isomorphic fundamental groups. The basic idea is that the vertices
of $K(E)$ are the edges of $GH(E)$, and the vertices of $GH(E)$
are, in some sense, the edges of $K(E)$. The process basically
``blows up'' the vertices of $GH(E)$ to introduce the edges of
$K(E)$, and then crushes the original edges of $GH(E)$ to points
to create the vertices of $K(E)$. The blow-up process introduces
the triangular 2-cells needed for $K(E)$, and the crushing process
turns the square 2-cells of $GH(E)$ into the square 2-cells of
$K(E)$. All of the topological facts used below may be found, for
example, in \cite{Hatcher, Spanier}. More precisely, in the
theorem below, we prove that $K(E)$ is the 2-skeleton of a complex
that is homotopy equivalent to $GH(E)$ and in particular, they
have isomorphic fundamental groups at each vertex.

\begin{Thm} $\pi_1(K(E),e)$ is isomorphic to $\pi_1(GH(E),\L_{e})$
for each $e \in E$.
\end{Thm}

\begin{proof}

The first step is to blow up each vertex $R$ or $L$ of $GH(E)$ to an
$n$-simplex, where $n$ is the valence of the vertex. Figure
\ref{Mark1} shows the essential details. The basic idea is that the
vertex $R$ or $L$ becomes the $n$-simplex, each edge of $GH(E)$
incident to $R$ or $L$ becomes an edge incident to a distinct vertex
of the $n$-simplex, and any square 2-cell incident to the vertex
receives an added edge of the $n$-simplex in its boundary, joining
the two vertices which its original pair of edges are now incident
to. Carrying out this process for all of the original vertices
results in a complex which we will call $Q_1$. Note that $Q_1$ is
homotopy equivalent to $GH(E)$, since $GH(E)$ may be obtained from
$Q_1$ by crushing each $n$-simplex $\sigma^n$ to a point (literally,
taking the quotient complex $Q_1/\sigma^n$). Since each $n$-simplex
is a contractible subcomplex of $Q_1$, the quotient map
$Q_1\rightarrow Q_1/\sigma^n$ is a homotopy equivalence
(\cite{Hatcher} Proposition 0.17); the result then follows by
induction, since $Q_1$ with every one of the introduced simplices
crushed to points is isomorphic to $GH(E)$. The original square
2-cells of $GH(E)$ have now become octagons in $Q_1$.

\begin{figure}\label{Mark1}

\includegraphics[width=0.7\textwidth]{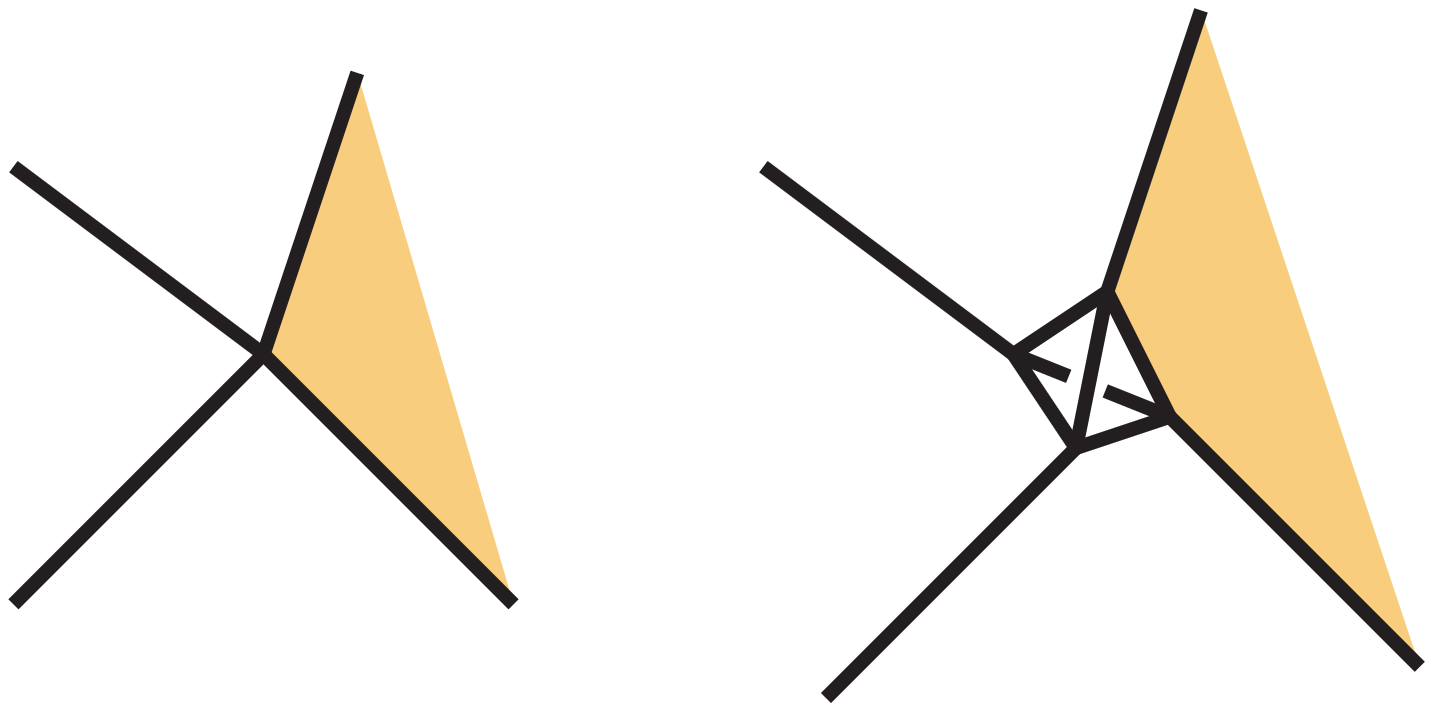}

\caption{}

\label{Mark1}

\end{figure}
\medskip

The complex $Q_1$ has a pair of vertices for each original edge of
$GH(E)$, that is, for each element $e\in E$. One of the vertices
lies in the 2-skeleton of the $n$-simplex corresponding to the
$\L$-class of $e$, and the other in the corresponding $\R$-class.
Our second step is to crush each of these original edges from
$GH(E)$ to points, resulting in a complex which we will call
$Q_2$; see Figure \ref{Mark2}. Each such edge forms a contractible
subcomplex of $Q_1$, since its vertices are distinct - the
1-skeleton of $GH(E)$ is a bipartite graph, so the vertices of
each edge lie on distinct $n$-simplices - so quotienting out by
each edge is again a homotopy equivalence. $Q_2$ is therefore
homotopy equivalent to $Q_1$. The vertices of $Q_2$ are now in
1-to-1 correspondence with $E$, since there is one vertex for each
edge in $GH(E)$. The edges of $Q_2$ are precisely the edges in the
$n$-simplices, so there is an edge from $e$ to $f$ precisely when
$e$ and $f$ lie in the same $\L$- or $\R$-class, which are
precisely the edges of the Nambooripad complex. Under the quotient
map the octagonal 2-cells of $Q_1$ have become square 2-cells,
whose boundaries are edge paths through the vertices $e,f,g,h$
given by the edges in the boundaries of the square 2-cells of
$GH(E)$. That is, they are precisely the singular $E$-squares of
the Nambooripad complex.

\leavevmode

\begin{figure} \label{Mark2}

\includegraphics[width=0.7\textwidth]{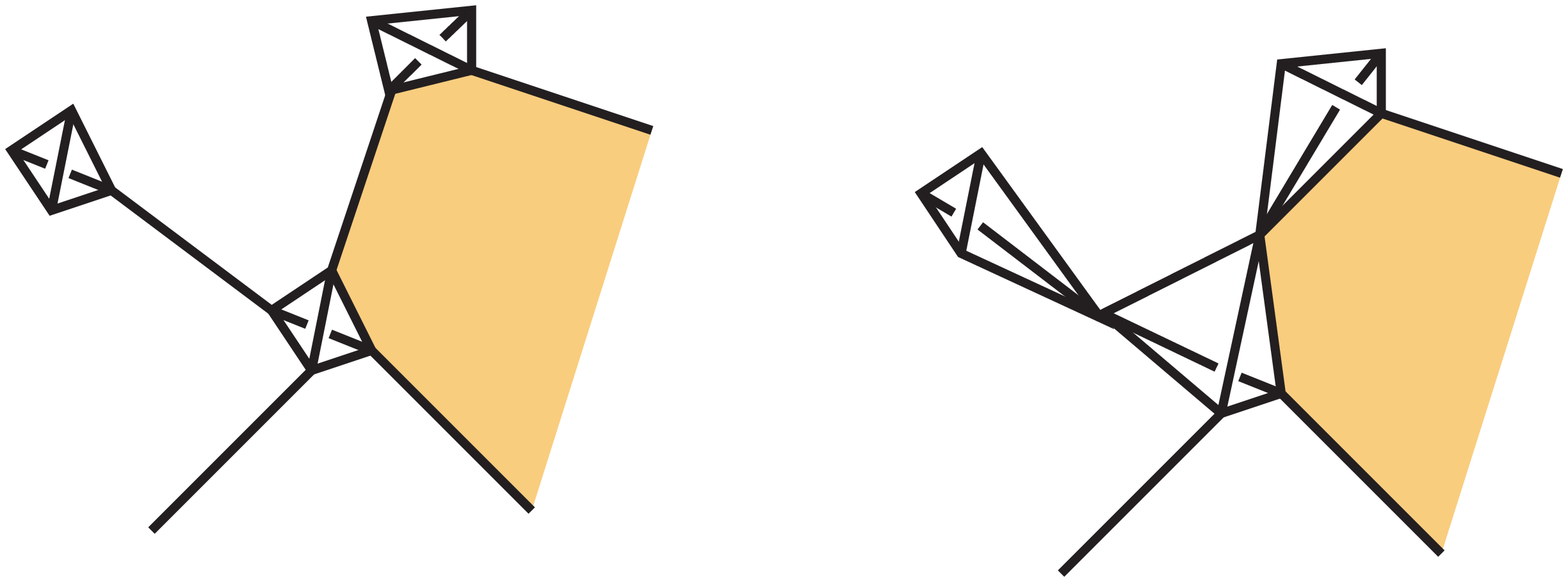}

\caption{}

\label{Mark2}
\end{figure}

Finally, the Nambooripad complex $K(E)$ is isomorphic to the
2-skeleton $Q_2^{(2)}\subseteq Q_2$ of $Q_2$. That is, $Q_2^{(2)}$
consists of the 1-skeleton, which is the 1-skeleton of $K(E)$,
together with the singular squares and all of the 2-faces of the
$n$-simplices, which are precisely the triangular 2-cells  of
$K(E)$ for $e,f,g$ three distinct elements in the same $\L$- or
$\R$-class. Having the same vertices, edges, and 2-cells, the two
2-complexes are therefore isomorphic.

Since the fundamental groupoid of the 2-skeleton of a complex is
isomorphic to the fundamental groupoid of the complex, we have

$\pi_1(K(E))\cong \pi_1(Q_2^{(2)})\cong \pi_1(Q_2)\cong
\pi_1(Q_1)\cong\pi_1(GH(E))$,

\noindent as desired.

\end{proof}

\section{An example of a free idempotent generated semigroup with
non-free subgroups}

In this section we present an example of a finite regular
biordered set $E$ such that $Z \times Z$, the free Abelian group
of rank 2, is isomorphic to a maximal subgroup of $RIG(E)$. This
is the first example of a subgroup of a free idempotent generated
semigroup that is not a free group.

Before presenting the example, we give more details on the
connection between bipartite graphs and completely 0-simple
semigroups. This will help us explain how we present our example.

Let $S=M^{0}(A,1,B,C)$ be a combinatorial completely 0-simple
semigroup. That is, the maximal subgroup is the trivial group 1.
Thus we can represent elements as pairs $(a,b) \in A \times B$ with
product $(a,b)(a',b')=(a,b')$ if $C(b,a')\neq 0$ and 0 otherwise. As
in the general case of the Graham-Houghton graph that we described
in the previous section, we associate a bipartite graph $\Gamma(S)$
to $S$. The vertices of $\Gamma(S)$ are $A \cup B$ (where as usual,
we assume $A \cap B$ is empty). There is an edge between $b \in B$
and $a \in A$ if and only if $C(b,a)=1$. Clearly $\Gamma(S)$ is a
bipartite graph with no isolated vertices.

Conversely, let $\Gamma$ be a bipartite graph with vertices the
disjoint union of two sets $A$ and $B$ and no isolated vertices.
We then have the incidence matrix $C=C(\Gamma):B \times A
\rightarrow \{0,1\}$ with $C(b,a)=1$ if and only if $\{b,a\}$ is
an edge of $\Gamma$. As usual we write $C$ as a $\{0,1\}$ matrix
with rows labelled by elements of $B$ and columns labelled by
elements of $A$. Define $S(\Gamma)$ to be the Rees matrix
semigroup $S(\Gamma)=M^{0}(A,1,B,C(\Gamma))$. Then it follows from
the fact that $\Gamma$ has no isolated vertices that $S(\Gamma)$
is a combinatorial 0-simple semigroup. Clearly, these assignments
give a one to one correspondence between combinatorial 0-simple
semigroups and directed bipartite graphs with no isolated
vertices. Isomorphisms of graphs are easily seen to correspond to
isomorphisms of the corresponding semigroup and vice versa.

We now explain the idea of our example. We will define a bipartite
graph $\Gamma$ that embeds on the surface of a torus. The graph will
represent the one skeleton of a square complex. We will then define
a finite regular semigroup $S$ that has $\Gamma$ as the bipartite
graph corresponding to a completely 0-simple semigroup that is an
ideal of $S$ and such that if we add the singular squares of the
biordered set $E(S)$ as 2-cells to $\Gamma$ (in the language of the
previous section, we build the Graham-Houghton complex), we obtain a
complex that has the fundamental group of the torus, that is, $Z
\times Z$ as maximal subgroup.

We begin by drawing the graph $\Gamma$ in Figure \ref{Gamma}.

\begin{figure}[!h]\label{Gamma}
\psfrag{l1 }[][]{$L_1$}
\psfrag{l2 }[][]{$L_2$}
\psfrag{l3 }[][]{$L_3$}
\psfrag{l4 }[][]{$L_4$}
\psfrag{l5 }[][]{$L_5$}
\psfrag{l6 }[][]{$L_6$}
\psfrag{l7 }[][]{$L_7$}
\psfrag{l8 }[][]{$L_8$}
%
%
\psfrag{r1 }[][]{$R_1$}
\psfrag{r2 }[][]{$R_2$}
\psfrag{r3 }[][]{$R_3$}
\psfrag{r4 }[][]{$R_4$}
\psfrag{r5 }[][]{$R_5$}
\psfrag{r6 }[][]{$R_6$}
\psfrag{r7 }[][]{$R_7$}
\psfrag{r8 }[][]{$R_8$}
\includegraphics[width=0.7\textwidth]{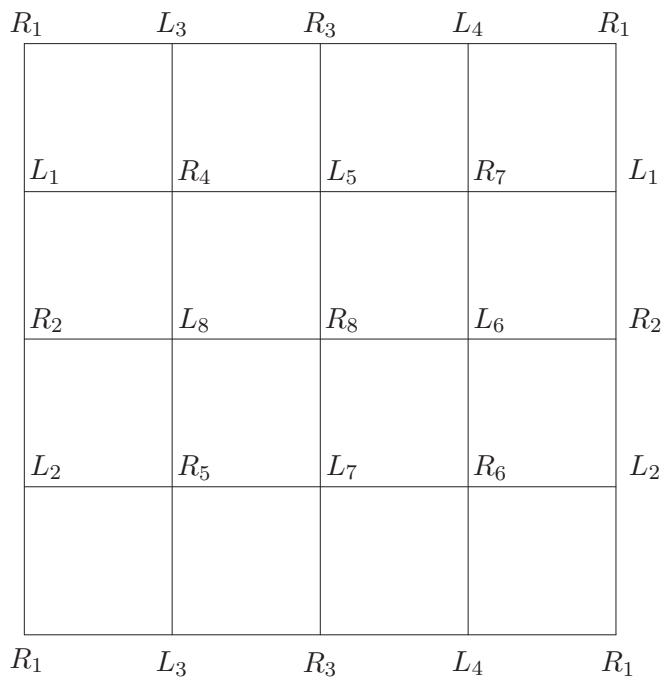}
\caption{The graph $\Gamma$}

\label{Gamma}
%
\end{figure}

We call the colors of the bipartition $R$ and $L$ to remind the
reader of the Green relations $\R$ and $\L$ (but if the reader
insists, s/he can think of them as Red and bLue). Thus there are
16 vertices in the graph and 32 edges. Figure \ref{Gamma} is drawn
in a way that the graph is really drawn on the torus obtained by
identifying the top of the graph with the bottom and the left side
with the right side.

Before continuing we define the incidence matrix of $\Gamma$. For
our purposes, it is more convenient to write the transpose of the
incidence matrix. Thus the matrix in figure \ref{incidence} has
rows labelled by $R_{1}, \ldots, R_{8}$ and columns labelled by
$L_{1}, \ldots, L_{8}$. In particular, the matrix written this way
defines the biordered set of the 0-simple semigroup $S(\Gamma)$
corresponding to $\Gamma$. That is, idempotents correspond to the
$\H$ classes with entries 1, the $\R$ relation corresponds to
being idempotents in the same row and the $\L$ relation
corresponds to being idempotents in the same column.

\begin{figure} \label{incidence}

$\left[%
\begin{array}{cccccccc}
  1 & 1 & 1 & 1 & 0 & 0 & 0 & 0 \\
  1 & 1 & 0 & 0 & 0 & 1 & 0 & 1 \\
  0 & 0 & 1 & 1 & 1 & 0 & 1 & 0 \\
  1 & 0 & 1 & 0 & 1 & 0 & 0 & 1 \\
  0 & 1 & 1 & 0 & 0 & 0 & 1 & 1 \\
  0 & 1 & 0 & 1 & 0 & 1 & 1 & 0 \\
  1 & 0 & 0 & 1 & 1 & 1 & 0 & 0 \\
  0 & 0 & 0 & 0 & 1 & 1 & 1 & 1 \\
\end{array}%
\right]$

\caption{The transpose of the incidence matrix of the graph
$\Gamma$}

\label{incidence}
\end{figure}

Now consider the 2-complex one obtains by sewing on 2-cells
corresponding to the 16 visual 1 by 1 squares that we see in the
diagram of $\Gamma$. Notice that after identifying the graph on the
surface of a torus, there are 24 4-cycles in the graph. There are
the 16 4-cycles bounding 2-cells in our complex (such as
$R_{1},L_{3},R_{4},L_{1}$) that we see in figure \ref{Gamma}: there
are also the 8 4-cycles (such as $R_{1},L_{3},R_{3},L_{4}$) that are
obtained when we fold $\Gamma$ into a torus, but these 4-cells do
not bound cells in our complex. Clearly the fundamental group of
this complex is $Z \times Z$. We have simply drawn subsquares on the
usual representation of the torus as a square with opposite sides
identified. By killing off these corresponding 16 4-cycles we have a
space homeomorphic to the torus and thus its fundamental group is $Z
\times Z$.

Furthermore, each of the 16 visual 1 by 1 squares in the diagram
of the graph $\Gamma$ corresponds to an $E$-square in the
biordered set of the 0-simple semigroup $S(\Gamma)$ corresponding
to $\Gamma$. Thus if we can find a regular semigroup $S$ that has
the biordered set corresponding to $S(\Gamma)$ as a connected
component and also has exactly the 16 visible squares as the
singular squares in this component, it follows from the results of
the previous section that the maximal subgroup of the connected
component corresponding to $\Gamma$ in $RIG(E(S))$ is $Z \times
Z$. We proceed to construct such a regular semigroup.

Let $X=\{L_{1}, \ldots, L_{8}\}$. The semigroup $S$ will be
defined as a subsemigroup of the monoid of partial functions
acting on the right of $X$. Let $C$ be the transpose of the matrix
in figure \ref{incidence}. Thus $C$ is the structure matrix of the
0-simple semigroup $S(\Gamma)$. To each element $s=(R_{i},L_{j})
\in S(\Gamma)$ we associate the partial constant function $f_s:X
\rightarrow X$ defined by $L_{x}f_{s}=L_{j}$ if $C(L_{x},R_{i})=1$
and undefined otherwise. In the language of semigroup theory,
$f_{s}$ is the image of $s$ under the right Schutzenberger
representation of $S(\Gamma)$ \cite{Arbib, qtheory}.

The semigroup generated by $\{f_{s}|s \in S(\Gamma)\}$ is isomorphic
to $S(\Gamma)$. This can  be verified by direct computation by
showing that for all $s,t \in S(\Gamma)$, $f_{s}f_{t}=f_{st}$,
(where $st$ is the product of $s$ and $t$ in $S(\Gamma)$) and that
the assignment $s \mapsto f_{s}$ is one to one. This follows
directly from the definition of $f_s$ above. Alternatively, one can
verify this by noting as we did above that the assignment of $s$ to
$f_{s}$ is the right Schutzenberger representation. The structure
matrix of $S(\Gamma)$, that is, the transpose of the matrix in
figure \ref{incidence}, has no repeated rows and columns and this
implies that both the right and left Schutzenberger representations
are faithful \cite{Arbib, qtheory}.

Now we define two more functions $e,k$ by the following tables.

\begin{center}
$e=\left[%
\begin{array}{cccccccc}
  L_{1} & L_{2} & L_{3} & L_{4} & L_{5} & L_{6} & L_{7} & L_{8} \\
  L_{1} & L_{6} & L_{3} & L_{7} & L_{3} & L_{6} & L_{7} & L_{1} \\
\end{array}%
\right]$

\medskip

$k=\left[%
\begin{array}{cccccccc}
  L_{1} & L_{2} & L_{3} & L_{4} & L_{5} & L_{6} & L_{7} & L_{8} \\
  L_{4} & L_{2} & L_{2} & L_{4} & L_{5} & L_{5} & L_{8} & L_{8} \\
\end{array}%
\right]$
\end{center}

Let $S$ be the semigroup generated by $\{e,k,f_{s}|s \in
S(\Gamma)\}$. We claim that $S$ is the semigroup that has the
properties we desire. Notice that $e$ and $k$ are idempotents and
that $S(\Gamma)$ is generated by its idempotents (this is known to
be equivalent to the graph $\Gamma$ being connected \cite{Gr,
Namb}), so in fact, $S$ is an idempotent generated semigroup.

The subsemigroup $T$ generated by $\{e,k\}$ has by direct
computation 8 elements
$\{e,k,(ek),(ke),(eke),(kek),h=(ek)^{2},f=(ke)^{2}\}$. This
semigroup consists of functions all of rank 4 and is a completely
simple semigroup whose idempotents are $e,f,k,h$. We claim that
$TS(\Gamma) \cup S(\Gamma)T \subseteq S(\Gamma)$. To see this we
first note that for $(R_{i},L_{j}) \in S(\Gamma)$, we have
$(R_{i},L_{j})t=(R_{i},L_{j}t)$ for $t \in \{e,k\}$. Therefore
$S(\Gamma)T \subseteq S(\Gamma)$ follows by induction on the length
of a product of elements in $\{e,k\}$.

We now list how $e$ and $k$ act on the left of $S(\Gamma)$. In the
charts below, we note, for $t \in \{e,k\}$ and $(R_{i},L_{j}) \in
S(\Gamma)$, that $t(R_{i},L_{j})=(tR_{i},L_{j})$ for the left action
$R_{i} \mapsto tR_{i}$ listed here. Again, all of this can be
verified by direct computation.

\begin{center}
$e:\left[%
\begin{array}{cccccccc}
 R_{1} & R_{2} & R_{3} & R_{4} & R_{5} & R_{6} & R_{7} & R_{8} \\
 R_{4} & R_{2} & R_{3} & R_{4} & R_{3} & R_{6} & R_{2} & R_{6} \\
\end{array}%
\right]$

\medskip

$k:\left[%
\begin{array}{cccccccc}
  R_{1} & R_{2} & R_{3} & R_{4} & R_{5} & R_{6} & R_{7} & R_{8} \\
  R_{1} & R_{5} & R_{7} & R_{8} & R_{5} & R_{1} & R_{7} & R_{8} \\
\end{array}%
\right]$
\end{center}

For readers who know the terminology, we have listed the images of
$T$ in the left Schutzenberger representation on $S(\Gamma)$
\cite{Arbib, qtheory}. Our claim that $TS(\Gamma) \cup S(\Gamma)T
\subseteq S(\Gamma)$ follows from these charts by induction on the
length of a product from $T$. It follows that $S$ is the disjoint
union of $T$ and $S(\Gamma)$. Thus $S$ is a regular semigroup with
3 $\J$ classes- one of them being $T$ and the other 2 coming from
$S(\Gamma)$ (its unique non-zero $\J$-class and 0). $S(\Gamma)$ is
the unique 0-minimal ideal of $S$. The order of $S$ is 73 and the
order of $E(S)$ is 37.

We now look at the biorder structure on $E(S)$. We summarize the
usual idempotent order relation in figure \ref{order}.

\begin{figure}[!h]\label{order}
\psfrag{l1 }[][]{$L_1$}
\psfrag{l2 }[][]{$L_2$}
\psfrag{l3 }[][]{$L_3$}
\psfrag{l4 }[][]{$L_4$}
\psfrag{l5 }[][]{$L_5$}
\psfrag{l6 }[][]{$L_6$}
\psfrag{l7 }[][]{$L_7$}
\psfrag{l8 }[][]{$L_8$}
%
%
\psfrag{r1 }[][]{$R_1$}
\psfrag{r2 }[][]{$R_2$}
\psfrag{r3 }[][]{$R_3$}
\psfrag{r4 }[][]{$R_4$}
\psfrag{r5 }[][]{$R_5$}
\psfrag{r6 }[][]{$R_6$}
\psfrag{r7 }[][]{$R_7$}
\psfrag{r8 }[][]{$R_8$}
\includegraphics[width=0.7\textwidth]{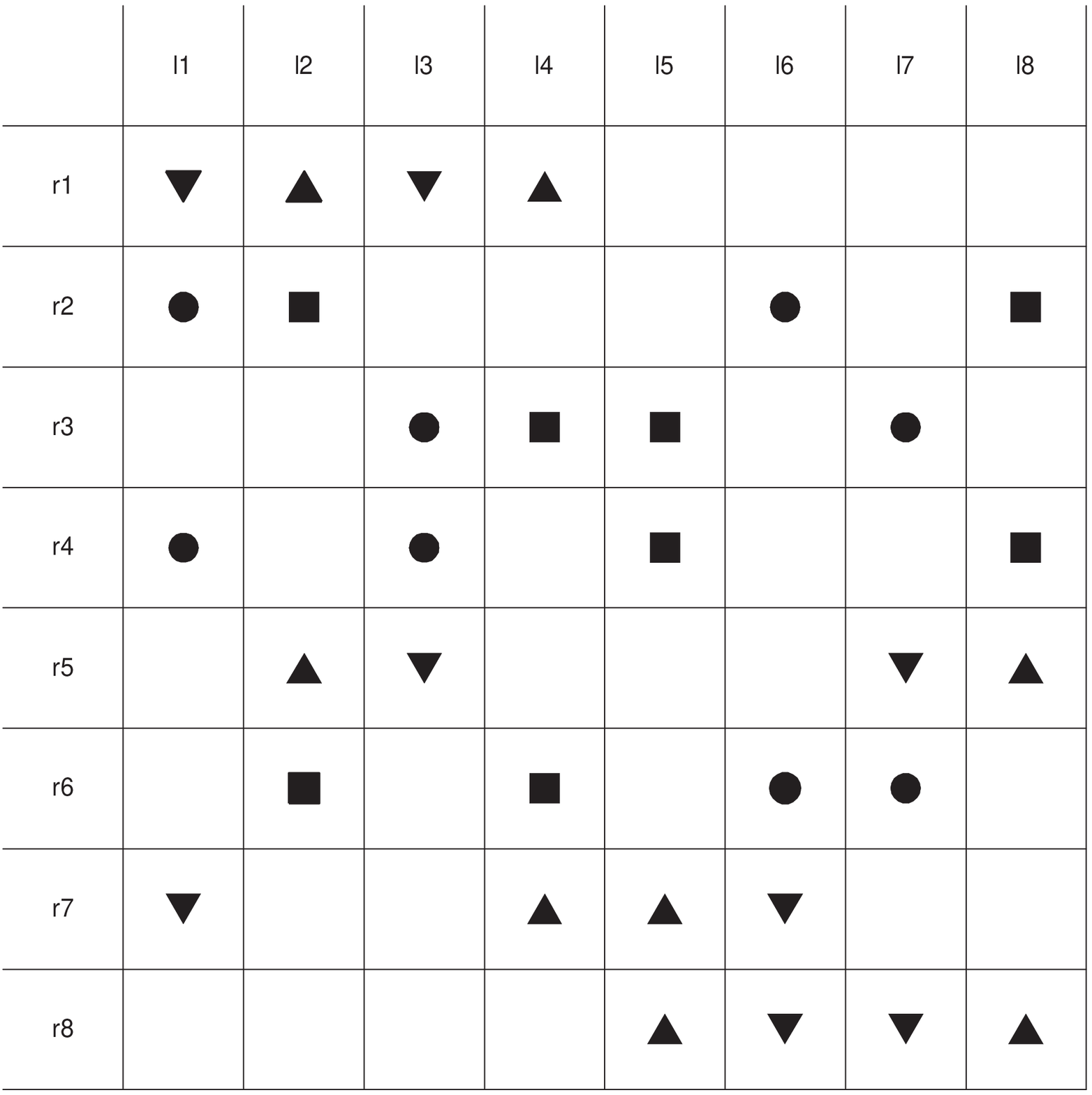}
\caption{The idempotent order on $E(S)$}

\label{order}
\end{figure}
We explain the symbols in this diagram. Each symbol represents an
idempotent in $T$ according to figure \ref{explanation}.
\begin{figure}\label{explanation}
\begin{center}
\begin{tabular}{|c|c|}
  \hline
  Idempotent & Symbol \\
  \hline
  $h$ & $\Box$ \\
  \hline
  $e$ & $\bullet$ \\
  \hline
  $k$ & $\triangle$ \\
  \hline
  $f$ & $\nabla$ \\
  \hline
\end{tabular}
\end{center}
\caption{} \label{explanation}
\end{figure}

An entry of a symbol in a box in figure \ref{order} denotes a
relation in the usual idempotent order. For example, the idempotent
$(R_{1},L_{1})$ of $S(\Gamma)$ is below $f$ in the idempotent order.
For example it follows from the diagram that $(R_{2},L_{1})\lb f$
but that $(R_{2},L_{1})$ is not below $f$ in the idempotent order.
The other relations in the regular biordered set $E(S)$ can be
computed directly in $S$. For example, $f(R_{2},L_{1}) =
(R_{7},L_{1}), k(R_{2},L_{2}) = (R_{5},L_{2})$, etc.

The partial order on $E(S)$ has many pleasant properties. For
example, each of the idempotents in $T$ is above exactly 8
idempotents in $S(\Gamma)$ and every idempotent in $S(\Gamma)$ is
below exactly one idempotent in $T$. The 8 idempotents in
$S(\Gamma)$ below a given idempotent in $T$ form an $E$-cycle.
Thus the idempotents in $S(\Gamma)$ decompose into the disjoint
union of 4 $E$-cycles of length 8. Below we give a more geometric
definition of the semigroup $S$ which will help explain some of
these properties.

Finally, in figure \ref{singularization}, we give the precise
information on which idempotents in $T$ singularize squares in
$E(S(\Gamma))$. Again, all of this can be verified by direct
computation.

\begin{figure}[!h]\label{singularization}
\psfrag{l1 }[][]{$L_1$}
\psfrag{l2 }[][]{$L_2$}
\psfrag{l3 }[][]{$L_3$}
\psfrag{l4 }[][]{$L_4$}
\psfrag{l5 }[][]{$L_5$}
\psfrag{l6 }[][]{$L_6$}
\psfrag{l7 }[][]{$L_7$}
\psfrag{l8 }[][]{$L_8$}
%
%
\psfrag{r1 }[][]{$R_1$}
\psfrag{r2 }[][]{$R_2$}
\psfrag{r3 }[][]{$R_3$}
\psfrag{r4 }[][]{$R_4$}
\psfrag{r5 }[][]{$R_5$}
\psfrag{r6 }[][]{$R_6$}
\psfrag{r7 }[][]{$R_7$}
\psfrag{r8 }[][]{$R_8$}
\includegraphics[width=0.7\textwidth]{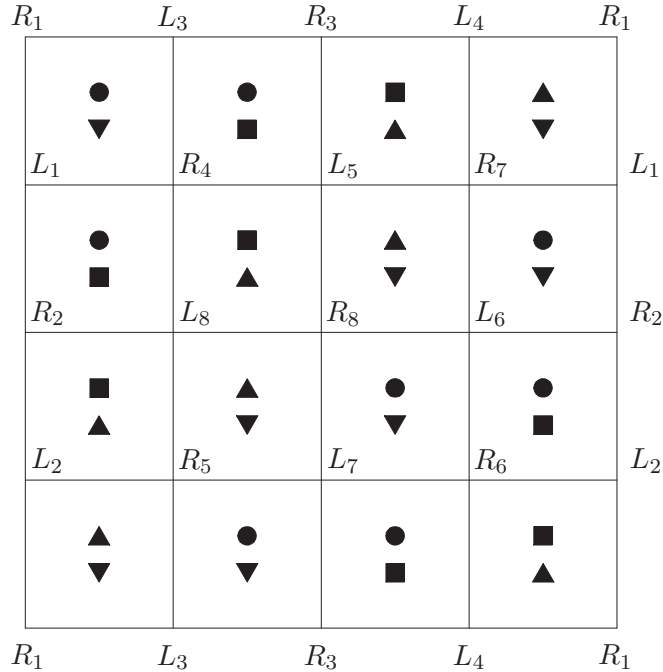}

\caption{Singularization of $E$-squares}

\label{singularization}
\end{figure}

The explanation of figure \ref{singularization} is as follows. An
entry in a square of the symbol of an idempotent from $T$
indicates that that idempotent singularizes the corresponding $2
\times 2$ rectangular set in $E(S)$. For example, the square,
$\left[%
\begin{array}{cc}
  (R_{1},L_{1}) & (R_{1},L_{3}) \\
  (R_{4},L_{1}) & (R_{4},L_{3}) \\
\end{array}%
\right]$, which is the square represented in the top left portion
of figure \ref{singularization} is singularized (bottom to top) by
$f$ and (top to bottom) by $e$. The diligent reader can verify all
that we claim by direct computation in $E(S)$. In particular,
exactly the 16 squares that we desire to be singularized in
$S(\Gamma)$ are the ones singularized in $S$ and therefore the
free (regular) idempotent semigroup on the biordered set $E(S)$
has $Z \times Z$ as a maximal subgroup for the connected component
corresponding to $\Gamma$ as explained at the beginning of this
section. This completes our first description of $S$. We now give
a more geometric description of the semigroup $S$.

\subsection{Incidence structures and
affine geometry over $Z_2$}

In this subsection we show that the semigroup $S$ discussed above
arises from a combinatorial structure related to affine 3-space
over $Z_2$. We first recall some connections between incidence
structures in the sense of combinatorics and finite 0-simple
semigroups.

Up to now, we have used the tight connection between bipartite
graphs and 0-simple semigroups over the trivial group to build our
example. As is well known, $\{0,1\}$-matrices arise naturally to
code information about other combinatorial structures besides
bipartite graphs.

An {\em incidence system} is a pair $D=(V,\mathcal{B})$ where $V$ is
a (usually finite) set of points and $\mathcal{B}$ is a list of
subsets of $V$ called {\em blocks}. We allow for the possibility
that a block, that is a certain subset of $V$, can appear more than
once in the list $\mathcal{B}$. The {\em incidence matrix} of $D$ is
the $|\mathcal{B}| \times |V|$ matrix $I_D$ (we will use the
elements of $\mathcal{B}$ and $V$ to name rows and columns) such
that $I_{D}(b,v) = 1$ if $v \in b$ and 0 otherwise, where $b \in
\mathcal{B}$ and $v \in V$. Sometimes, the transpose of this matrix
is called the incidence matrix, but it is more convenient for our
purposes to define things this way.

The semigroup $S(D)$ associated with $D$ is the Rees matrix
semigroup $M^{0}(\mathcal{B},1,V,C)$ where $C$ is the transpose of
$I_{D}$. It is straightforward to see that $S(D)$ is 0-simple if
and only if the empty set is not a block and every point belongs
to some block. We make these assumptions throughout. Conversely,
it is easy to see that the transpose of the structure matrix of a
combinatorial completely 0-simple semigroup is an incidence system
with these two properties.

For example, if we consider the matrix in figure \ref{incidence}
as an incidence system, the points are $\{L_{1},\ldots,L_{8}\}$.
The blocks are
$R_{1}=\{L_{1},L_{2},L_{3},L_{4}\},R_{2}=\{L_{1},L_{2},L_{6},L_{8}\}$,etc.

Now we show that this incidence system can be coordinatized as a
certain affine configuration over the field of order 2 and that the
semigroup $S$ can be faithfully represented by affine partial
functions that are ``continuous" with respect to this structure in
the sense of \cite{DM1, DM2, DM3}.

Let $F_2$ be the field of order 2 and let $V=F_2^{3}$ be 3-space
over $F_{2}$. Consider the set of planes through the origin (i.e.
2 dimensional subspaces of $V$) that do not contain the vector
$(1,1,1)$.  An elementary counting argument shows that there are 4
such planes. We let $\mathcal{B}$ be the set of these 4 planes
plus their 4 translates by the vector (1,1,1). Therefore,
$\mathcal{B}$ has 8 elements. We claim that by suitably ordering
the points in $V$ and the planes in $\mathcal{B}$, the incidence
matrix of $(V,\mathcal{B})$ is the matrix in figure
\ref{incidence}. We do this by making the assignment of vectors to
the points $L_{1}, \ldots , L_{8}$ according to figure
\ref{points}.

\medskip
\begin{figure}\label{points}
\begin{center}
\begin{tabular}{|c|c|}

  \hline
  $L_{1}$ & (0,0,0) \\
  \hline
   $L_{2}$ & (1,0,0) \\
  \hline
  $L_{3}$ & (0,1,0) \\
  \hline
  $L_{4}$ & (1,1,0) \\
  \hline
  $L_{5}$ & (0,1,1) \\
  \hline
  $L_{6}$ & (1,0,1) \\
  \hline
  $L_{7}$ & (1,1,1) \\
  \hline
  $L_{8}$ & (0,0,1) \\
  \hline
\end{tabular}
\end{center}

\caption{The points of the structure}

\label{points}
\end{figure}
\medskip

With this identification of the $L_{i}$ as vectors in $V$, we have
the following way to identify the blocks of our structure. For
simplicity of presentation, we write $i$ in place of $L_{i}$ in
figure \ref{blocks}.

\medskip

\begin{figure}\label{blocks}
\begin{center}
\begin{tabular}{|c|c|c|}
  \hline
  Row & Block & Subset of $V$ \\
  \hline
  $R_1$ & $\{1,2,3,4\}$ & $\{(0,0,0),(1,0,0),(0,1,0),(1,1,0)\}$ \\
  \hline
  $R_2$ & $\{1,2,6,8\}$ & $\{(0,0,0),(1,0,0),(1,0,1),(0,0,1)\}$ \\
  \hline
  $R_3$ & $\{3,4,5,7\}$ & $\{(0,1,0),(1,1,0),(0,1,1),(1,1,1)\}$ \\
  \hline
  $R_4$ & $\{1,3,5,8\}$ & $\{(0,0,0),(0,1,0),(0,1,1),(0,0,1)\}$ \\
  \hline
  $R_5$ & $\{2,3,7,8\}$ & $\{(1,0,0),(0,1,0),(1,1,1),(0,0,1)\}$ \\
  \hline
  $R_6$ & $\{2,4,6,7\}$ & $\{(1,0,0),(1,1,0),(1,0,1),(1,1,1)\}$ \\
  \hline
  $R_7$ & $\{1,4,5,6\}$ & $\{(0,0,0),(1,1,0),(0,1,1),(1,0,1)\}$ \\
  \hline
  $R_8$ & $\{5,6,7,8\}$ & $\{(0,1,1),(1,0,1),(1,1,1),(0,0,1)\}$ \\
  \hline
\end{tabular}
\end{center}

\caption{The blocks of the structure}

\label{blocks}
\end{figure}
\medskip

We can see from the preceding table that $R_{1},R_{2},R_{4},R_{7}$
are precisely the 4 planes through the origin in $V$ that do not
contain the vector $(1,1,1)$ and that
$R_{3}=R_{2}+(1,1,1),R_{5}=R_{7}+(1,1,1),R_{6}=R_{4}+(1,1,1),R_{8}=R_{1}+(1,1,1)$
are their translates.

Now we show that the semigroup $S$ defined in the previous
subsection also has a natural interpretation with respect to this
geometric structure. Let $V$ be a vector space over an arbitrary
field. An affine partial function on $V$ is a partial function
$f_{A,w}:V \rightarrow V$ of the form $vf=vA+w$,where $A:V
\rightarrow V$ is a partial linear transformation, that is a
linear transformation whose domain is an affine subspace of $V$
and range an affine subspace of $V$ and $w \in V$. The collection
of all affine partial functions is a monoid $Aff(V)$. If we
identify $f_{A,w}$ with the pair $(A,w)$, then multiplication in
$Aff(V)$ takes the form $(A,w)(A',w')=(AA',wA'+w')$ so that
$Aff(V)$ is a semidirect product of the monoid of partial linear
transformations on $V$ with the additive group on $V$.

We claim that the idempotents $e$ and $k$ defined in the previous
section in defining our semigroup $S$ act as affine functions on
$F_{2}^3$ using our translation of our structure in this section.
Indeed, let $A=\left[%
\begin{array}{ccc}
  1 & 0 & 1 \\
  0 & 1 & 0 \\
  0 & 0 & 0 \\
\end{array}%
\right]$ considered as a matrix over $F_2$. Then it is easily
checked that for $1 \leq i \leq 8$, $ie=j$ if and only if
$v_{i}A=v_{j}$ where $v_{i}$ is the vector corresponding to
$L_{i}$ in the table above and that if $B=\left[%
\begin{array}{ccc}
  0 & 1 & 0 \\
  0 & 1 & 0 \\
  1 & 1 & 1 \\
\end{array}%
\right]$ and $w=(1,1,0)$, then for $1 \leq i \leq 8$, $ik=j$ if
and only if $v_{i}B + w =v_{j}$. Thus the completely simple
subsemigroup $T$ of our semigroup $S$ is faithfully represented by
affine functions over our geometric structure.

Furthermore, each element of $T$ has the following property with
respect to this structure: the inverse image of each plane in the
structure is also in the structure. For example, $R_{1}e^{-1}=
R_{4}, R_{2}e^{-1}= R_{2}, R_{3}e^{-1}=R_{3}, R_{4}e^{-1}=R_{4},
R_{5}e^{-1}=R_{3}, R_{6}e^{-1}=R_{6}, R_{7}e^{-1}=R_{2},
R_{8}e^{-1}=R_{6}$.

Each element $(R_{i},L_{j})$ is also represented as an affine
partial function, namely the partial function whose domain is
$R_{i}$ and sends all points in its domain to $L_{j}$. We can
represent this as an affine partial function by taking $A$ to be
the 0 linear transformation restricted to $R_{i}$ and $w$ to be
$L_j$. Clearly, the inverse image of a block $R$ under this
function is either $R_{i}$ if $L_{j} \in R$ and the empty set
otherwise.

Notice also, that for every element of $S$ the closure of blocks
under inverse image encodes left multiplication of $e$ in the
biordered set $E(S)$. For example,
$e(R_{1},L_{1})=(R_{1}e^{-1},L_{1})=(R_{4},L_{1})$,
$(R_{1},L_{1})(R_{3},L_{1})=0$, etc.

Thus, there is an analogue of the action of the partial functions
on our structure to continuous functions on a topological space.
If we consider the blocks of our structure to be ``open", then our
functions preserve open sets under inverse image. The notion of
continuous partial functions on combinatorial structures and its
relationship to the semigroup theoretic notion of translational
hull \cite{CP} has been explored in \cite{DM1,DM2,DM3}. We see
here that there is a close connection between building biordered
sets with a specific connected component and the continuous
partial functions on the corresponding 0-simple semigroup. We will
explore this connection in future work.

\section{Summary and future directions}

We have shown how to represent the maximal subgroups of the free
(regular) idempotent generated semigroup on a regular biordered set
by a 2-complex derived from Nambooripad's \cite{Namb} work. By
applying the Bass-Serre techniques of \cite{HMM}, we are directly
lead to the graph defined by Graham and Houghton for finite 0-simple
semigroups \cite{Gr, Hough}. We put a structure of 2-complex on this
graph and use that to construct an example of a finite regular
biordered set that has a maximal subgroup that is isomorphic to the
free abelian group of rank 2. This is the first example of a
non-free group that appears in a free idempotent generated
semigroup.

The biordered set arises from a certain combinatorial structure
defined on a 3 dimensional vector space over the field of order 2.
This suggests looking for further examples by either varying the
field and looking at analogous structures over 3 dimensional
spaces or by looking at higher dimensional analogues of the
structure we have defined.

In related work we have proved, using completely different
techniques, that if $F$ is any field, and $E_{3}(F)$ is the
biordered set of the monoid of $3 \times 3$ matrices over $F$, then
the free idempotent generated semigroup over $E_{3}(F)$ has a
maximal subgroup isomorphic to the multiplicative subgroup of $F$.
In particular, finite cyclic groups of order $p^{n}-1$, $p$ a prime
number appear as maximal subgroups of free idempotent generated
semigroups.

This last example motivates an intended application of this work. We
would like to apply Nambooripad's powerful theory of inductive
groupoids \cite{Namb} to study reductive linear algebraic monoids
\cite{Putchbook}. This very important class of regular monoids and
their finite analogues have been intensively studied over the last
25 years. A basic example is the monoid of all matrices over a
field.

The above discussion begs the question of describing the class of
groups that are maximal subgroup of $IG(E)$ or $RIG(E)$ for a
biordered set $E$. This seems to be a very difficult question at
this time.

\section{Acknowledgements}

The authors would like to thank Noel Brady for discussions about an
earlier version of this work and Ludmilla Epstein-Marcus for her
help in drawing the figures 3-7.

\end{document}